\documentclass[12pt]{amsart}
\usepackage{amsmath,amssymb,amsthm}
\usepackage[all]{xy}

\newcommand{\Gr}{\mathrm{Gr}}
\newcommand{\Irr}{\operatorname{Irr}}
\newcommand{\IC}{\mathrm{IC}}
\newcommand{\CS}{\mathcal{CS}}
\newcommand{\Crys}{\mathcal{C}rys}
\newcommand{\cC}{\mathcal{C}}
\newcommand{\Hom}{\operatorname{Hom}}
\def\O{\mathcal{O}}
\def\K{\mathcal{K}}
\def\cD{\mathcal{D}}
\def\C{\mathbb{C}}
\def\ttt{\tilde{\times}}

\newtheorem{Theorem}{Theorem}[section]
\newtheorem{Proposition}[Theorem]{Proposition}

\title{A combinatorial geometric Satake equivalence}
\author{Joel Kamnitzer}
\email{jkamnitz@math.toronto.edu}
\address{Department of Mathematics\\ University of Toronto \\ Toronto, ON Canada}

\begin{document}
\begin{abstract}
The geometric Satake correspondence provides an equivalence of categories between the Satake category of spherical perverse sheaves on the affine Grassmannian and the category of representations of the dual group.  In this note, we define a combinatorial version of the Satake category using irreducible components of fibres of the convolution morphism.  We then prove an equivalence of coboundary categories between this combinatorial Satake category and the category of crystals of the dual group.
\end{abstract}

\maketitle

\section{Introduction}
Let $ G $ be a complex reductive group and let $ G^\vee $ be its Langlands dual group.

The geometric Satake equivalence of Lusztig \cite{L}, Ginzburg \cite{G}, and Mirkovic-Vilonen \cite{MV} provides a description of the representation theory of $G^\vee $ in terms of the topology of the affine Grassmannian $ \Gr = G((t))/G[[t]] $ of $ G $.  More precisely, the above authors defined a symmetric monoidal category of $ G[[t]]$-equivariant perverse sheaves on $ \Gr $ (known as the Satake category) and then proved that the Satake category is equivalent to the category of representations of $ G^\vee $.

In this paper, we define a combinatorial version of the Satake category and then prove that it is equivalent to the category of $ G^\vee $-crystals.

\subsection{The combinatorial Satake category}
To explain our combinatorial Satake category, let us recall that the usual Satake category is a semisimple category whose simple objects are the IC sheaves $ \IC(\Gr^\lambda) $ of spherical Schubert varieties.  The monoidal structure is defined by convolution and a standard computation shows that $$ \IC(\Gr^\lambda) \otimes \IC(\Gr^\mu) \cong \bigoplus_{\nu} IC(\Gr^\nu) \otimes H_{2\langle \lambda + \mu - \nu, \rho \rangle}(m^{-1}(t^\nu)), $$
where $ m : \Gr^\lambda \ttt \Gr^\mu \rightarrow \Gr $ is the convolution morphism.  The vector space $H_{2\langle \lambda + \mu - \nu, \rho \rangle}(m^{-1}(t^\nu)) $ has a natural basis consisting of the set $C_{\lambda \mu}^\nu $ of top-dimensional irreducible components of $ m^{-1}(t^\nu) $.

Thus, we combinatorialize the Satake category by defining a semisimple monoidal category $ \CS $ where the tensor product is defined using the sets $C_{\lambda \mu}^\nu $.  We then equip this category with associativity and commutativity constraints.  For associativity, we use iterated convolutions, and for commutativity, we use a certain automorphism of $ G $, inspired by an idea of Beilinson-Drinfeld \cite{BD}.

\subsection{The equivalence with crystals}
Having defined this combinatorial version of the Satake category, it is natural to compare it to the category of $G^\vee$-crystals, which is a combinatorial version of the representation category of $G^\vee$.  We are able to prove an equivalence between these two categories using the work of Braverman-Gaitsgory \cite{BG}.

\begin{Theorem} \label{th:main}
There is an equivalence of coboundary categories $ \CS \cong G^\vee$-$\Crys $.
\end{Theorem}

It is not immediately obvious that $\CS$ is a coboundary category, but this follows from the theorem.  The coboundary category structure of $\CS $ will be further explored in \cite{GKS}.

We should emphasize that though we think of $ G^\vee$-$\Crys$ as a combinatorial version of $ Rep \, G^\vee $, it is genuinely different.  More precisely, suppose we form $ G^\vee$-$\Crys \otimes \C $; the category where the objects are the same as $ G^\vee$-$\Crys$ and where the morphism sets have been $\C$-linearly extended.  Then $ G^\vee$-$\Crys \otimes \C $ is certainly equivalent to $ Rep \, G^\vee $ as a category, but it is not equivalent as a monoidal category, not even for $ G^\vee = SL_2$, as can be seen by considering the $6j$-symbols.  Likewise, the monoidal category $ \CS $ is genuinely different from the usual Satake category.

\subsection{Acknowledgements}
I would like to thank Xinwen Zhu for an email discussion which prompted me to complete this paper.  I would also like to thank A. Goncharov and A. Henriques for useful discussions.

\section{Background}
\subsection{Notation}

Let $ G $ denote a semisimple group.  Fix a maximal torus $ T $ and let $ N $, $N_- $ denote unipotent radicals of opposite Borels containing $ T $.  We use $ X $ to denote the coweight lattice of $ G$.  Let $X_+ $ denote the dominant coweights of $ G $.  We use $ W $ for the Weyl group of $ G $ and we use $ w_0 $ for its longest element.  We write $ \lambda^* := -w_0 \lambda $ and we write $ i^* $ if $ \alpha_i^* = \alpha_{i^*} $.

Let $ \psi : G \rightarrow G $ denote the involutive automorphism of $ G $ which is defined by $ \psi(s) = s^{-1} $ for $ s \in T $ and $ \psi(x_\alpha(a)) = x_{-\alpha}(-a) $ for all roots $ \alpha $ (here $ x_\alpha : \mathbb{G}_a \rightarrow G $ is the subgroup corresponding to the root $ \alpha$).  In other words, $ \psi(g) = (g^{t})^{-1} $.

Let $ G^\vee $ denote the Langlands dual group to $ G $.  Recall that $ X_+ $ is the set of dominant weights of $ G^\vee $.

\subsection{Affine Grassmannian}
Let $ \K = \C((t)) $, $ \O = \C[[t]] $ denote the field of Laurent series and the ring of power series.  Then the affine Grassmannian of $ G $ is defined as $ \Gr = G(\K)/G(\O) $.

For each $ \mu \in X $, we have the point $ t^\mu \in \Gr $.

We have the distance function $ d : \Gr \times \Gr \rightarrow X_+ $ whose level sets are the $ G(\K) $-orbits.  In particular $ d(1, t^\lambda) = \lambda $ if $ \lambda $ is dominant.

For $\lambda $ dominant, we let $ Gr^\lambda = G(\O)t^\lambda = \{ L : d(1, L) = \lambda \} $ be the usual spherical Schubert cell.  Recall that $ \dim \Gr^\lambda = \langle 2\lambda, \rho \rangle $.  We also let $ S^\mu := N(\K) t^\mu $ be the semi-infinite cell and let $ T^\mu := N_-(\K)t^\mu $ be the opposite semi-infinite cell.

If $ \lambda_1, \dots, \lambda_n $ are dominant weights, then we can form the convolution variety
$$
\Gr^{\lambda_1} \ttt \cdots \ttt \Gr^{\lambda_n} = \{ (L_1, \dots, L_n) : d(L_{i-1}, L_i) = \lambda_i \text{ for all } i\}
$$
where we interpret $ L_0 = 1$.

We write $ m_{\lambda_1, \dots, \lambda_n} : \Gr^{\lambda_1} \ttt \cdots \ttt \Gr^{\lambda_n} \rightarrow \Gr $ for the map $(L_1, \dots, L_n) \mapsto L_n $.

In this paper, we will be interested in the varieties
$$
m_{\lambda_1, \dots, \lambda_n}^{-1}(t^\mu), \ m_{\lambda_1, \dots, \lambda_n}^{-1}(\Gr^\mu), \ m_{\lambda_1, \dots, \lambda_n}^{-1}(T^\mu)
$$
There varieties have dimensions at most $ \langle \lambda - \mu, \rho \rangle, \langle \lambda + \mu, \rho \rangle, \langle \lambda - \mu, \rho \rangle$, respectively, where $ \lambda = \lambda_1 + \dots + \lambda_n $ and where in the first two expressions we assume that $ \mu $ is dominant.

We will write
\begin{gather*}
C_{\lambda_1 \dots \lambda_n}^\mu := \Irr m_{\lambda_1, \dots, \lambda_n}^{-1}(t^\mu),  \ \Irr m_{\lambda_1, \dots, \lambda_n}^{-1}(\Gr^\mu), \\ MV(\lambda_1, \dots, \lambda_n)_\mu := \Irr m_{\lambda_1, \dots, \lambda_n}^{-1}(T^\mu)
\end{gather*}
for the sets of irreducible components of these dimensions.

Note that $ m_{\lambda_1, \dots, \lambda_n}^{-1}(\Gr^\mu) $ is a bundle over $ \Gr^\mu $ with fibre $ m_{\lambda_1, \dots, \lambda_n}^{-1}(t^\mu)$.  Thus there is a bijection $$ C_{\lambda_1 \dots \lambda_n}^\mu \rightarrow \Irr m_{\lambda_1, \dots, \lambda_n}^{-1}(\Gr^\mu)$$ which we write as $ Z \mapsto \widetilde Z $.  Concretely, this bijection is realized by $ \widetilde Z = G(\O) Z $.

\subsection{Monoidal categories} \label{se:moncat}
Recall that a monoidal category is a category $ \cC $ along with a bifunctor $ \otimes : \cC \times \cC \rightarrow \cC$ and an isomorphism $ \alpha : \otimes \circ (\otimes \times I) \rightarrow \otimes \circ (I \times \otimes) $ which is required to satisfy the pentagon axiom.   The isomorphism $\alpha $ is called the associativity constraint or associator.

\subsubsection{Refinement of the associator}
We can repackage the data of $ \alpha $ in the following way.  First, we suppose the existence of a trifunctor $ \otimes_3 : \cC \times \cC \times \cC \rightarrow \cC $ and then we assume that we can write $ \alpha $ as a composition $ \alpha = \alpha^2 \circ (\alpha^1)^{-1}$ where $ \alpha^1 : \otimes_3 \rightarrow \otimes \circ (\otimes \times I) $ and $ \alpha^2 : \otimes_3 \rightarrow \otimes \circ (I \times \otimes) $.  In other words, for every three objects $ A, B, C$ of $ \cC $, we have an unbracketed triple tensor product $ A \otimes B \otimes C $ and isomorphisms
$$
\alpha^1_{A, B,C} : A \otimes B \otimes C \rightarrow (A \otimes B) \otimes C   \quad \alpha^2_{A, B, C} : A \otimes B \otimes C \rightarrow A \otimes (B \otimes C)
$$
with $$\alpha_{A, B, C} = \alpha^2_{A, B, C} (\alpha^1_{A, B, C})^{-1} : (A \otimes B) \otimes C \rightarrow   A \otimes (B \otimes C) $$
We will always work with monoidal categories which come equipped with this extra structure.

A monoidal functor $ \Phi : \cC \rightarrow \cD $ a functor $ \Phi $ along with a natural isomorphism $ \phi: \otimes \circ (\Phi \times \Phi)  \rightarrow \Phi \circ \otimes $ such that for any three objects $ A, B, C$ of $ \cC $, the following diagram commutes
\begin{equation*}
\xymatrix{
(\Phi(A) \otimes \Phi(B)) \otimes \Phi(C) \ar[r]^{\alpha_{\Phi(A), \Phi(B), \Phi(C)}} \ar[d]^{\phi_{A \otimes B, C} \circ (\phi_{A,B} \otimes I)} & \Phi(A) \otimes (\Phi(B) \otimes \Phi(C)) \ar[d]^{\phi_{A, B \otimes C} \circ (I \otimes \phi_{B,C})}\\
\Phi((A \otimes B) \otimes C) \ar[r]^{\Phi(\alpha_{A,B,C})} & \Phi(A \otimes (B \otimes C))
}
\end{equation*}

In the presence of the unbracketed triple tensor product $ \otimes_3 $, it is natural to ask for another natural isomorphism $ \phi : \otimes_3 \circ F \times F \times F  \rightarrow F \circ \otimes_3 $.  We demand that the following diagram commutes
\begin{equation} \label{eq:monfunc}
\xymatrix{
\Phi(A) \otimes \Phi(B) \otimes \Phi(C) \ar[r]^{\alpha^1_{\Phi(A), \Phi(B), \Phi(C)}} \ar[d]^{\phi_{A, B, C}} & (\Phi(A) \otimes \Phi(B)) \otimes \Phi(C) \ar[d]^{\phi_{A \otimes B, C} \circ (\phi_{A,B} \otimes I)}\\
\Phi(A \otimes B \otimes C) \ar[r]^{F(\alpha^1_{A,B,C})} & \Phi((A \otimes B) \otimes C)
}
\end{equation}
and an analogous one for $ \alpha^2 $.  These commutative diagrams imply the above one involving $ \alpha$.

\subsubsection{Commutativity}
A commutativity constraint (or commutor) on a monoidal category is a natural isomorphism $ \sigma : \otimes \rightarrow \otimes^{op} $, in other words for any two objects $A, B$, we have an isomorphism $$ \sigma_{A, B} : A \otimes B \rightarrow B \otimes A. $$

If $ \cC, \cD $ are monoidal categories with commutors, then a monoidal functor $ \Phi : \cC \rightarrow \cD $ is compatible with the commutors if
\begin{equation} \label{eq:commutors}
\xymatrix{
\Phi(A) \otimes \Phi(B) \ar[r]^{\sigma_{\Phi(A), \Phi(B)}} \ar[d]^{\phi_{A,B}} & \Phi(B) \otimes \Phi(A) \ar[d]^{\phi_{B,A}} \\
\Phi(A \otimes B) \ar[r]^{\Phi(\sigma_{A,B})} & \Phi(B \otimes A) \\
}
\end{equation}
commutes.

There are a number of additional axioms that one can impose on a commutor.  If $ \sigma $ is symmetric ($\sigma_{A,B} = \sigma_{B,A}^{-1} $) and satisfies a certain hexagon axiom involving the associator, then we say that $ \cC $ is a coboundary category (see \cite{HK1} for more details).

\subsection{Category of crystals}
\subsubsection{Crystals}
A crystal $ B $ is a set $ B $ equipped with maps $$ e_i : B \rightarrow B \cup \{0\}, \ f_i : B \rightarrow B \cup \{0 \}, \text{for $ i \in I$, and } \mathrm{wt} : B \rightarrow X. $$  
This data is required to satisfy certain natural axioms (see for example \cite{HK1}).

For each dominant weight $ \lambda $, there is a crystal $ B(\lambda) $, which is a combinatorial version of the representation $ V(\lambda) $.  Let $ G^\vee$-$\Crys $ denote the category of $G^\vee$-crystals, as defined in \cite{HK1}.  This category consists of all crystals isomorphic to a direct sum of crystals $ B(\lambda) $ where $ \lambda $ ranges over the dominant weights of $ G^\vee $.  A morphism of crystals is simply a map of sets which commutes will all crystal operations.

A highest weight element $ b $ of a crystal $ B $ is an element which satisfies $ e_i(b) = 0 $ for all $ i \in I $.  Note that a map of crystals $ f : B \rightarrow C $ is completely determined by its value on the highest weight elements of $ B$.  The crystal $ B(\lambda) $ has a unique highest weight element $ b_\lambda $.

\subsubsection{Tensor products}

The category $ G^\vee$-$\Crys$ has a monoidal structure, as defined by Kashiwara.  If $ B, C $ are two crystals, then we write $ B \otimes C $ for their tensor product (which is just $ B \times C $ as a set).

If $ B $ carries a trivial crystal structure (all elements are of weight $0 $), then $ B \otimes C $ is just a direct sum of different copies of $ C $ labelled by $ B $.

The associator for the category of crystals is the same as the associator for the category of sets.  Thus, it is safe to ignore it and identify $ (A \otimes B) \otimes C = A \otimes B \otimes C$.

\subsubsection{The crystal commutor}
In \cite{HK1}, we defined a commutor for the category of crystals, giving $G^\vee$-$\Crys$ the structure of a coboundary category.  For the purposes of this paper, we will not need the original definition of the commutor, but rather the following description from \cite{KT}.

Recall that we can embed (not as a morphism of crystals, but compatible with all $ e_i$) the crystal $ B(\lambda) $ into the crystal $B(\infty) $ of a Verma module.  We write this embedding as
$$ \iota_\lambda : B(\lambda) \hookrightarrow B(\infty) .$$
The crystal $ B(\infty) $ carries the Kashiwara involution $ * $.  The main result from \cite{KT} describes $ \sigma $ using $ \iota $ and $ * $.  As explained before it is enough to give the image of all highest weight elements under $ \sigma $.  The definition of the tensor product ensures that all highest weight elements in $ B(\lambda) \otimes B(\mu) $ are of the form $ b \otimes b_\mu $.
\begin{Theorem} \label{th:peter}
Let $ b \otimes b_\mu $ be a highest weight element of $ B(\lambda) \otimes B(\mu) $.  Then $ \sigma(b \otimes b_\mu) = b' \otimes b_\lambda $ where $ \iota_{\mu}(b') = \iota_\lambda(b)^*$.
\end{Theorem}

\section{Construction of the combinatorial Satake category}

\subsection{The combinatorial Satake category}
As discussed in the introduction, we will consider a semisimple category  $\CS $ whose simple objects are labelled by $ \lambda \in X_+ $ and where tensor product multiplicities are given by $ C_{\lambda \mu}^\nu = \Irr m_{\lambda \mu}^{-1}(t^\nu)$.

More precisely, an object $ R $ of $ \CS $ is a choice of set $ R_\lambda $ for each $ \lambda \in X_+ $.  Morphisms in the category are defined by
$$
\Hom(R, S) = \prod_{\lambda} \Hom(R_\lambda, S_\lambda).
$$
We define $ R \otimes S $ by
$$
(R \otimes S)_\nu = \bigcup_{\lambda, \mu} R_\lambda \times S_\mu \times C_{\lambda \mu}^\nu.
$$

We define an unbracketed triple tensor product by
$$
(R \otimes S \otimes T)_\delta = \bigcup_{\lambda, \mu, \nu} R_\lambda \times S_\mu \times T_\nu \times C_{\lambda \mu \nu}^\delta.
$$
Let $ A(\lambda) $ denote the object which assigns the set $ \{1\} $ to $ \lambda $ and $\emptyset$ to every other weight.  These $ A(\lambda) $ are the simple objects and if $ R $ is any object of $ \CS $, then we have an isomorphism $ R \cong \oplus_\lambda A(\lambda)^{\oplus |R_\lambda|}$.

\subsection{Associativity constraint}
Our goal is to define an associativity constraint for the category $ \CS $.  As explained in section \ref{se:moncat}, we will actually define two associators $ \alpha^1, \alpha^2 $.

From the definition of the category, it suffices to define isomorphisms
\begin{equation*}
\begin{gathered}
\label{eq:unbracket}
\alpha^1 :  A(\lambda) \otimes A(\mu) \otimes A(\nu) \rightarrow (A(\lambda) \otimes A(\mu)) \otimes A(\nu) , \\ \alpha^2 :  A(\lambda) \otimes A(\mu) \otimes A(\nu) \rightarrow A(\lambda) \otimes (A(\mu) \otimes A(\nu))
\end{gathered}
\end{equation*}

We will concentrate on the construction of $ \alpha^1 $ as $ \alpha^2 $ is similar.  The multiplicity of $ A(\delta) $ in the left hand side of $ \alpha^1 $ is given by
$$
\bigcup_\gamma C_{\lambda \mu}^\gamma \times C_{\gamma \nu}^\delta
$$
Thus, we must construct a bijection
$$
Q_{\lambda \mu \nu}^\delta : \Irr m_{\lambda \mu \nu}^{-1}(t^\delta) \rightarrow \bigcup_\gamma \Irr m_{\lambda \mu}^{-1}(t^\gamma) \times \Irr m_{\gamma \nu}^{-1}(t^\delta)
$$
Now, let us study $ m_{\lambda \mu \nu}^{-1}(t^\delta)$.  For each $ \gamma \in X_+$, we define $$ m_{\lambda \mu \nu}^{-1}(t^\delta)_\gamma = \{ (L_i) \in m_{\lambda \mu \nu}^{-1}(t^\delta) : L_2 \in \Gr^\gamma\}.$$  Clearly, this gives us a decomposition into locally closed subsets $$m_{\lambda \mu \nu}^{-1}(t^\delta) = \sqcup_\gamma m_{\lambda \mu \nu}^{-1}(t^\delta)_\gamma.  $$
For each irreducible component $ Y \in \Irr m_{\lambda \mu \nu}^{-1}(t^\delta) $, there exists some $ \delta $ such that $ Y \cap m_{\lambda \mu \nu}^{-1}(t^\delta)_\gamma $ is dense in $ Y $.  For this $ Y, \delta $, we see that $ Y \cap m_{\lambda \mu \nu}^{-1}(t^\delta)_\gamma $ is an irreducible component of $ m_{\lambda \mu \nu}^{-1}(t^\delta)_\gamma $.

Now this locus $ m_{\lambda \mu \nu}^{-1}(t^\delta)_\gamma$ can be described as a fibre product,
$$m_{\lambda \mu \nu}^{-1}(t^\delta)_\gamma = m_{\lambda \mu}^{-1}(\Gr^\gamma) \times_{\Gr^\gamma} m_{\gamma \nu}^{-1}(t^\delta). $$
Thus each irreducible component of $m_{\lambda \mu \nu}^{-1}(t^\delta)_\gamma $ of dimension $ \langle \lambda + \mu + \nu - \delta, \rho \rangle $ is of the form $ \widetilde Y_1 \times_{\Gr^\gamma} Y_2 $ for $ Y_1 \in C_{\lambda \mu}^\gamma$ and $ Y_2 \in C_{\gamma \nu}^\delta $.  Here we use that $$ \dim \widetilde Y_1 \times_{\Gr^\gamma} Y_2 = \langle \lambda + \mu + \gamma, \rho \rangle + \langle \gamma + \nu - \delta, \rho \rangle - \langle \gamma, 2 \rho \rangle = \langle \lambda + \mu + \nu - \delta, \rho \rangle. $$

Thus we define a bijection
\begin{equation*}
\begin{aligned}
Q_{\lambda \mu}^\nu : \Irr m_{\lambda \mu \nu}^{-1}(t^\delta) &\rightarrow \bigcup_\gamma \Irr m_{\lambda \mu}^{-1}(t^\gamma) \times \Irr m_{\gamma \nu}^{-1}(t^\delta) \\
Y &\mapsto (Y_1, Y_2)
\end{aligned}
\end{equation*}
where $ \gamma $ is chosen so that $ Y \cap m_{\lambda \mu \nu}^{-1}(t^\delta)_\gamma $ is dense in $ Y $ and $$  Y \cap m_{\lambda \mu \nu}^{-1}(t^\delta)_\gamma = \widetilde Y_1 \times_{\Gr^\gamma} Y_2.$$

To conclude, we have constructed a monoidal category $ \CS $ where the tensor product is defined using the sets $ C_{\lambda \mu}^\nu $.  (Actually, we have not checked the pentagon axiom, though it follows by considering 4-fold convolutions.  On the other hand, it also follows from the main theorem.)

\subsection{Commutativity constraint}
To define a commutativity constraint, we need to construct a bijection $ C_{\lambda \mu}^\nu \rightarrow C_{\mu \lambda}^\nu $.  In fact we will define an isomorphism $ m_{\lambda \mu}^{-1}(t^\nu) \rightarrow m_{\mu \lambda}^{-1}(t^\nu) $.  Our definition is inspired by Beilinson-Drinfeld \cite[5.3.8]{BD}.

Recall the automorphism $\psi $ of $ G $.  Extend $ \psi $ to an automorphism of $ G(\K) $ and of $ \Gr $.  Note that $ \psi(t^\lambda) = t^{-\lambda} $ and note that $ d(\psi(L_1), \psi(L_2)) = d(L_1, L_2)^* = d(L_2, L_1) $.

\begin{Proposition}
If $ L \in m_{\lambda \mu}^{-1}(t^\nu) $, then $ t^\nu \psi(L) \in m_{\mu \lambda}^{-1}(t^\nu) $ and this defines an isomorphism $ t^\nu \psi : m_{\lambda \mu}^{-1}(t^\nu) \rightarrow m_{\mu \lambda}^{-1}(t^\nu) $.
\end{Proposition}

\begin{proof}
Since $L \in m_{\lambda \mu}^{-1}(t^\nu) $, we have that $ d(1, L) = \lambda $.  Thus $$d(t^\nu \psi(L), t^\nu) = d(\psi(L), 1) = \lambda. $$  A similar reasoning shows that $ d(1, t^\nu \psi(L)) = \mu $.  This proves the first part of the proposition.  For the second part, it is easy to see that $ L \mapsto t^\nu \psi^{-1}(L) $ is an inverse.
\end{proof}

Thus we have constructed the bijection $ C_{\lambda \mu}^\nu \rightarrow C_{\mu \lambda}^\nu $ as desired.

\section{The equivalence}

\subsection{Construction of the equivalence}
We begin by defining a functor $ \Phi : \CS \rightarrow G^\vee$-$\Crys $ as follows.

Recall that Braverman-Gaitsgory \cite{BG} have defined a crystal structure on the set of MV cycles for $ \Gr^\lambda $ and proved that this realizes the crystal $ B(\lambda) $.  Let us write $ MV(\lambda)_\mu :=  \Irr \Gr^\lambda \cap T^\mu $ for the  MV cycles of weight $ \mu $ and $ MV(\lambda) = \cup_\mu MV(\lambda)_\mu$.

We define the functor $ \Phi $ to take the object $ A(\lambda) $ in $ \CS $ to the crystal $ MV(\lambda) $.  More precisely, we define
$$
\Phi(R) := \bigcup_{\lambda \in X_+} R_\lambda \otimes MV(\lambda)
$$
where $ R_\lambda $ carries a trivial crystal structure.

Now we would like to extend $ \Phi $ to a monoidal functor.  To do this, we have to produce a natural transformation $ \phi_{R,S} : \Phi(R) \otimes \Phi(S) \rightarrow \Phi(R \otimes S) $.  Thus we need to construct an isomorphism of crystals
$$
(\bigcup_\lambda R_\lambda \otimes MV(\lambda)) \otimes (\bigcup_\mu S_\mu \otimes MV(\mu)) \rightarrow
\bigcup_\nu \bigcup_{\lambda, \mu} (R_\lambda \times S_\mu \times C_{\lambda \mu}^\nu) \otimes MV(\nu)
$$
For this it suffices to define for each $ \lambda, \mu, $ an isomorphism of crystals
$$
P_{\lambda \mu} : MV(\lambda) \otimes MV(\mu) \rightarrow \bigcup_\nu C_{\lambda \mu}^\nu \otimes MV(\nu)
$$

Fortunately, Braverman-Gaitsgory \cite[Theorem 3.2]{BG} have constructed such an isomorphism of crystals.  Following their paper, we recall that the map $P_{\lambda \mu} $ is defined as follows.  First, given $ Z_1 \in MV(\lambda)_\tau $ and $ Z_2 \in MV(\mu)_\delta $, we can form
$$
Z_1 \ttt Z_2 := \{ ([g], L_2) \in \Gr \times \Gr : [g] \in Z_1, g \in N_-(\K)t^\tau, g^{-1}L_2 \in Z_2 \}
$$
This is well-defined since $ Z_2 $ is invariant under $ N_-(\O)$.  Note that $Z_1 \ttt Z_2 \in MV(\lambda, \mu)_{\tau + \delta} $ and that this defines a bijection $ MV(\lambda) \times MV(\mu) \rightarrow MV(\lambda, \mu) $.

Now we will define $P_{\lambda \mu} $ by examining $ Z_1 \ttt Z_2 $ in a different way.  First choose $ \nu $ such that  that $ Z_1 \ttt Z_2 \cap m^{-1}(\Gr^\nu) $ is dense in $ Z_1 \ttt Z_2 $.  Then define $ P_{\lambda \mu}(Z_1, Z_2) = (Y, Z) \in C_{\lambda \mu}^\nu \times MV(\nu) $ so that $$ \widetilde Y \times_{\Gr^\nu} Z = (Z_1 \ttt Z_2) \cap m^{-1}(\Gr^\nu). $$

In a similar way we define
$$
P_{\lambda \mu \nu} : MV(\lambda) \otimes MV(\mu) \otimes MV(\nu) \rightarrow \bigcup_{\gamma} C_{\lambda \mu \nu}^\gamma \otimes MV(\gamma)
$$
by $ P_{\lambda \mu \nu}(Z_1, Z_2, Z_3) = (Y, Z) $ if $$ \widetilde Y \times_{\Gr^\gamma} Z = (Z_1 \ttt Z_2 \ttt Z_3) \cap m^{-1}(\Gr^\gamma) $$ (assuming that intersection is dense in $ Z_1 \ttt Z_2 \ttt Z_3 $).  The argument from \cite{BG} can be extended to show that $ P_{\lambda \mu \nu} $ is an isomorphism of crystals (it is easy to see that it is a bijection of sets), but this fact will actually follow from analysis in the next section.  Assuming for the moment that $P_{\lambda \mu \nu}$ is a map of crystals, then we see that it gives to an isomorphism $ \phi_{R,S,T} : \Phi(R) \times \Phi(S) \times \Phi(T) \rightarrow \Phi(R \otimes S \otimes T) $ for any three objects $R,S,T \in \CS $.

From the definition, it is obvious that $ \Phi $ is an equivalence of categories. Thus, Theorem \ref{th:main} comes down to the following result.
\begin{Theorem} \label{th:compat}
The pair $(\Phi, \phi) $ is compatible with associators and commutors.
\end{Theorem}

\subsection{Compatibility of associators}
We consider the diagram \eqref{eq:monfunc}.  We will just check this diagram for $ \alpha^1 $ as the one for $ \alpha^2 $ will follow by similar reasoning.  As usual it is enough to check this diagram when the three objects are $ A(\lambda), A(\mu), A(\nu) $.  Then applying the definitions we are reduced to the following diagram.
\begin{equation} \label{eq:assoc}
\xymatrix{
MV(\lambda) \times MV(\mu) \times MV(\nu) \ar[r]^{P_{\lambda \mu} \times I} \ar[d]^{P_{\lambda \mu \nu}} & \bigcup_\delta C_{\lambda \mu}^\delta \times MV(\delta) \times MV(\nu)  \ar[d]^{I \times P_{\delta \nu}} \\
\bigcup_\gamma C_{\lambda \mu \nu}^\gamma \times MV(\gamma) \ar[r]^{Q_{\lambda \mu \nu}^\gamma} & \bigcup_{\gamma} \bigcup_\delta C_{\lambda \mu}^\delta \times C_{\delta \nu}^\gamma \times MV(\gamma)
}
\end{equation}
This diagram is a slight modification of \eqref{eq:monfunc}; we have omitted the associator in the crystal category and split the composition $  \phi \circ (\phi \otimes I) $ into two maps.

Let us now check the commutativity of the diagram.  Let $ (Z_1, Z_2, Z_3) \in MV(\lambda) \times MV(\mu) \times MV(\nu)$.  The commutativity of this diagram depends on alternate descriptions of the convolution $Z_1 \ttt Z_2 \ttt Z_3$.  First, let us choose $ \delta, \gamma $ such that the two loci
$$ \{ (L_i) \in Z_1 \ttt Z_2 \ttt Z_3 : L_2 \in \Gr^\delta \} \text{ and } \{ (L_i) \in Z_1 \ttt Z_2 \ttt Z_3 : L_3 \in \Gr^\gamma \}$$
are both dense in $ Z_1 \ttt Z_2 \ttt Z_3$.  For simplicity of notation below, we write $ D_1, D_2 $ for these two dense subsets of $ Z_1 \ttt Z_2 \ttt Z_3 $.

By definition $ P_{\lambda \mu \nu}(Z_1, Z_2, Z_3) = (Y, Z) $ where $ D_1  = \widetilde Y \times_{\Gr^\gamma} Z$.  Also, $ Q_{\lambda \mu \nu}^\gamma(Y) = (Y_1, Y_2) $ where $$ Y \cap m_{\lambda \mu \nu}^{-1}(t^\gamma)_\delta = \widetilde{Y_1} \times_{\Gr^\delta}  Y_2.$$  Thus following the diagram down and right takes $ (Z_1, Z_2, Z_3) $ to $ (Y_1, Y_2, Z) $ where
$$ D_1 \cap D_2 = \widetilde{ \widetilde{Y_1} \times_{\Gr^\delta} Y_2} \times_{\Gr^\gamma} Z $$

On the other hand following the diagram right and then down gives $(Y'_1, Y'_2, Z') $ where
$$
D_1 \cap D_2 = \widetilde{Y'_1} \times_{\Gr^{\delta}} (\widetilde{Y'_2} \times_{\Gr^{\gamma}} Z')
$$

However, examining the definitions we see that $$ \widetilde{ \widetilde{Y_1} \times_{\Gr^\delta} Y_2} \times_{\Gr^\gamma} Z = \widetilde{Y_1} \times_{\Gr^\delta} (\widetilde{Y_2} \times_{\Gr^\gamma} Z)$$
and thus we conclude that $ (Y_1, Y_2, Z) = (Y'_1, Y'_2, Z') $ as desired.

Note that we never used that $ P_{\lambda \mu \nu} $ is a crystal morphism, though it now follows (since the rest of the maps in \eqref{eq:assoc} are crystal morphisms).

\subsection{Compatibility of commutors}
We must check the commutativity of the following diagram.
\begin{equation*}
\xymatrix{
MV(\lambda) \times MV(\mu) \ar[r]^{\sigma_{MV(\lambda), MV(\mu)}} \ar[d]^{P_{\lambda \mu}} & MV(\mu) \times MV(\lambda) \ar[d]^{P_{\mu \lambda}} \\
\bigcup_\nu C_{\lambda \mu}^\nu \times MV(\nu) \ar[r]^{t^\nu \psi \times I} & \bigcup_\nu C_{\mu \lambda}^\nu \times MV(\nu) \\
}
\end{equation*}

Since all the maps in this diagram are crystal morphisms, it suffices to check the commutativity of this diagram on highest weight elements.  So let us take a highest weight element $ (Z, \{t^\mu\}) $.  We see that it suffices to prove the following result.

\begin{Proposition}
The action of the crystal commutor on this highest weight element is given by $ \sigma(Z, \{t^\mu \}) = (t^\nu\psi(Z), \{t^\lambda\}) $ where $ Z $ has weight $ \nu - \mu $.
\end{Proposition}
\begin{proof}
In light of Theorem \ref{th:peter}, it suffices to show that $$ \iota_\lambda(Z)^* = \iota_{\mu}(t^\nu \psi(Z)). $$

Let
$$
MV(\infty) = \bigcup_\gamma \Irr \overline{S^0 \cap T^\gamma}
$$
where $ \gamma $ ranges over the negative root cone.  By Braverman-Finkelberg-Gaisgory \cite{BFG}, there is a crystal structure on $ MV(\infty) $ making it isomorphic to the crystal $ B(\infty) $.  Moreover, they show that Kashiwara's involution is given by $ Z^* = t^\gamma \psi(Z)$ (see also \cite{BaGa}).

The map $ \iota_\lambda : B(\lambda) \rightarrow B(\infty) $ can be realized in this model as follows.  First, we note that because $ \lambda $ is dominant, there is an inclusion $ \Gr^\lambda \subset \overline{S^\lambda} $ and thus $ \Gr^\lambda \cap T^\mu \subset \overline{S^{\lambda} \cap T^\mu} $.   Moreover, these two varieties have the same dimension.  Thus if $ Z \in \Irr \Gr^\lambda \cap T^\mu$, we can form $ \overline{Z} \in \Irr \overline{S^{\lambda} \cap T^\mu} $.  We then multiply by $ t^{-\lambda} $ which defines an isomorphism between $ \overline{S^\lambda \cap T^\mu} $ and $ \overline{S^0 \cap T^{\mu - \lambda}} $.  By Proposition 4.3 of Baumann-Gaussent \cite{BaGa}, we have $ \iota_\lambda(Z) = t^{-\lambda} \overline{Z} $.

Now let us take $ Z $ as in the statement of the Proposition.  By hypothesis $ Z \in MV(\lambda)_{\nu - \mu} $.  Thus, we see that $$ \iota_\lambda(Z)^* = t^{\nu - \mu - \lambda} \psi(t^{-\lambda} \overline{Z}) = t^{\nu - \mu} \psi(\overline{Z}). $$

On the other hand, $$\iota_{\mu}(t^\nu \psi(Z)) = t^{-\mu}\overline{t^\nu \psi(Z)} = t^{\nu - \mu} \psi(\overline{Z}).$$
Thus the result follows.
\end{proof}

This completes the proof of Theorem \ref{th:compat} and thus Theorem \ref{th:main}.

\end{document}